\documentclass[12pt]{amsart}
\usepackage{fullpage}
\usepackage{adjustbox}
\usepackage{ae,amsfonts,euscript,enumerate}
\usepackage{amsmath,euscript}
\usepackage{amsthm}
\usepackage{comment}
\usepackage{amssymb,array}
\usepackage{enumitem}
\usepackage{booktabs}

\usepackage[usenames]{color}

\usepackage[colorlinks=true,
linkcolor=webgreen,
filecolor=webbrown,
citecolor=webgreen]{hyperref}
\definecolor{webgreen}{rgb}{0,.5,0}
\definecolor{webbrown}{rgb}{.6,0,0}

\usepackage{url}
\usepackage{graphics}
\usepackage{latexsym}
\usepackage{epsf}

\newtheorem{theorem}{Theorem}[section]
\newtheorem{lemma}[theorem]{Lemma}
\newtheorem{corollary}[theorem]{Corollary}

\newtheorem{proposition}[theorem]{Proposition}

\newtheorem{question}[theorem]{Question}

\theoremstyle{definition}

\newtheorem{remark}[theorem]{Remark}

\begin{document}

\author{Jason Bell}
\address[J. Bell]{Department of Pure Mathematics \\
University of Waterloo\\
Waterloo, ON  N2L 3G1 \\
Canada}
\email[J. Bell]{jpbell@uwaterloo.ca}
\thanks{Research of the first-named author supported by NSERC Grant RGPIN-2016-03632.}

\author{Marni Mishna}
\address[M. Mishna ]{Department of Mathematics  \\
Simon Fraser University\\
Burnaby, BC  V5A 1S6 \\
Canada}
\email[M. Mishna]{mmishna@sfu.ca}
\thanks{Research of the second-named author supported by NSERC Grant RGPIN-2017-04157}

\title{On the complexity of the cogrowth sequence}

\begin{abstract} Given a finitely generated group with generating set $S$, we study the \emph{cogrowth} sequence, which is the number of words of length $n$ over the alphabet $S$ that are equal to one.  This is related to the probability of return for walks in a Cayley graph with steps from $S$.  We prove that the cogrowth sequence is not $P$-recursive when~$G$ is an amenable group of superpolynomial growth, answering a question of Garrabant and Pak.  \end{abstract}

\maketitle

\section{Introduction}
One of the fundamental problems in the theory of finitely generated groups is the word problem, which asks whether there is an algorithm for determining when a word in a symmetric set of generators is equal to the identity.  For many classes of groups the word problem is decidable (i.e., there exists a decision procedure for determining if a word on a set of generators is equal to the identity). This includes free groups, one-relator groups, polycyclic groups and fundamental groups of closed orientable two-manifolds of genus greater than one.  On the other hand, there exist groups with unsolvable word problem, with the first such example being given by Novikov~\cite{Nov}.  

Algorithmically, given a finitely generated group $G$ with a symmetric generating set $S=S^{-1}$, the word problem is to find the elements in the free monoid on $S$  whose images in $G$ are equal to the identity.  In general, a \emph{language} over a finite alphabet $\Sigma$ is just a subset of the free monoid $\Sigma^*$ on $\Sigma$.  In our setting, the set of all possible sequences of letters from $S$ is denoted $S^*$, and then the set of words in $S*$ that correspond to elements whose image in $G$ is equal to the identity is a a sublanguage of $S^*$, denoted $\mathcal{L}(G;S)$. The complexity of the language  $\mathcal{L}(G;S)$ gives a hint to the difficulty of the word problem for $G$, and ultimately reveals information about $G$.  From a computational point of view, there is a coarse hierarchy, formulated by Chomsky, which says that sublanguages of $S^*$ can be divided into five nested classes roughly corresponding to the complexity of the machine needed to correctly decide if a string in $S^*$ is the language or not. The simplest class are the regular languages, which can be recognized by finite state-automata. These are contained in the set of context-free languages, recgonized by pushdown automata. These are contained in the set of context-sensitive languages (those recognized using linear-bounded non-deterministic Turing machines), in turn contained in recursively enumerable languages  (those that are recognized using Turing machines). There exists non-recursively enumerable languages so there is a strictly larger class containing of all possible languages. Readers unfamiliar with this classification are referred to Sipser \cite{Sipser} for more detailed definitions of the classes and their computational significance. 

% in general languages (including those that lie outside of the realm of classical computation); recursively enumerable languages (those that are recognized using Turing machines\footnote{We do not give the definition of a Turing machine and instead refer the reader to {\bf REF} ???. We point out, to aid with the reader's intuition, that Turing machines encompass everything that lies in the realm of classical computation, in the sense that anything that is computable by an algorithm can be computed by a Turing machine.}); context-sensitive languages (those produced using linear-bounded non-deterministic Turing machines); context-free languages (those produced using pushdown automata); and regular languages (those produced using finite-state automata).  

Adopting this point of view, it is natural to ask which finitely generated groups $G$ have the property that $\mathcal{L}:=\mathcal{L}(G,S)$ lies in a given class of the Chomsky hierarchy.  The answer is now understood for all classes except for context-sensitive.  It is known that $\mathcal{L}$ is regular if and only if $G$ is finite \cite{Ani}; $\mathcal{L}$ is context-free if and only if $G$ has a finite-index free subgroup (see Muller and Schupp~\cite{MS} with a missing ingredient supplied by work of Dunwoody~\cite{Dun}); $\mathcal{L}$ is recursively enumerable if and only if $G$ embeds in a finitely presented group~\cite{Hig}.  As mentioned, groups with context-sensitive word problem have not been classified, but an important result in this direction is that finitely generated subgroups of automatic groups have context-sensitive word problem~\cite{Shap}. These classes provide a taxonomy of groups in terms of the complexity of their word problem.

Given a sublanguage~$\mathcal{L}$ of a free monoid~$S^*$, one can often also gain insight into $\mathcal{L}$ by looking at the generating function for the number of words in $\mathcal{L}$ of length $n$. Generating functions have their own parallel containment hierarchy. The simplest combinatorial power series are expansions of rational functions; they are contained in the algebraic power series (those that generate finite extensions of the field of rational functions); next there are $D$-finite series which satisfy a linear homogeneous differential equation with rational function coefficients; after this comes the differentiably algebraic, or ADE, series (series satisfying algebraic differential equations), which have the property that the function and all of its derivatives generate a finite transcendence degree extension of the field of rational functions; all of these are contained in the set of general power series. 

Remarkably, these two hierarchies correspond to some extent on the lower end. It is natural to expect that simpler languages (with simplicity being understood in terms of where the language fits into the Chomsky hierarchy) should have well behaved generating functions compared to more complex languages and this is generally the case, up until one hits context-sensitive languages.  The simplest of these four classes in the Chomsky hierarchy is the collection of regular languages, which have rational generating series.  Context-free languages have algebraic generating series by a theorem of Chomsky and Sch\"utzenberger~\cite[Chapter III]{KS}.  At this point, the relationship between complexity of the language and complexity of the corresponding generating function breaks down.  For example, there are context-sensitive languages with non-differentiably algebraic generating functions. At the same time, a standard example of a context-sensitive language, $\{a^nb^nc^n: n\in \mathbb{N}\}$, has a rational generating function. 

Nonetheless, there are striking correspondences between finitely generated groups and the hierarchy of formal series. For a group~$G$ with generating set~$S$ we consider the ordinary generating function of the language $\mathcal{L}(G,S)$.  In this setting, the number of words of length~$n$ in $\mathcal{L}(S,G)$ is called the \emph{cogrowth} of $G$ with respect to $S$, and we denote it ${\rm CL}(n;G,S)$.  This is the \emph{cogrowth sequence of $G$}, and its ordinary generating function,  the series
\[F(t):= \sum {\rm CL}(n;G,S)t^n,\] is the \emph{cogrowth series of $G$ with respect to the generating set $S$\/}. 

Our main result concerns the cogrowth series for amenable groups (see \S2 for a definition).  Amenable groups are in some sense well-behaved and the class includes all solvable groups and groups of subexponential growth.  On the other hand, groups containing a free subgroup of rank two are non-amenable.  Our first main result is to show that the cogrowth series of an amenable group is never $D$-finite, except possibly when the group has polynomially bounded growth, answering a conjecture of Garrabant and Pak~\cite[Conjecture 13]{GP}.
\begin{theorem} \label{thm:main1} Let $G$ be a finitely generated amenable group that is not nilpotent-by-finite and let $S$ be a finite symmetric generating set for $G$.  Then the cogrowth series
$$\sum {\rm CL}(n;G,S)t^n$$ is not $D$-finite.  Equivalently, if $p_{G,S}(n)$ denotes the probability of return, then this sequence is not $P$-recursive.
\end{theorem}
A celebrated theorem of Gromov says that finitely generated groups of polynomially bounded growth are precisely those groups that are virtually nilpotent.  In particular, Theorem~\ref{thm:main1} says that amenable groups with superpolynomial growth do not have $P$-recursive cogrowth sequence.  There are many examples of non-amenable groups with $D$-finite cogrowth generating series. Virtually free groups have algebraic cogrowth generating series~\cite{MS} and a result of Elder, Janse van Rensburg, Rechnitzer, and Wong~\cite{ER} shows that the cogrowth series for the Baumslag-Solitar group $BS(N,N)$, which is non-amenable, has $D$-finite cogrowth generating series. Elvey Price and Guttman~\cite{ElGu} use sophisticated numerical techniques on the cogrowth sequence of Thompson's group $F$ to demonstrate that the asymptotics are likely incompatible with the group being amenable. 

The strategy behind the proof of Theorem~\ref{thm:main1} is to use the fact that if the cogrowth generating series is $D$-finite then it must be a $G$-function.  Then, using work of Chudnovsky and Chudnovsky, Katz, and Andr\'e \cite{CC}, \cite{Katz}, \cite{And1} we get that the generating series must have a very specific expansion in an open set of the form $U\setminus L$, where $U$ is an open neighbourhood of a singularity and $L$ is a ray emanating from the singularity.  This expression, combined with work of Kesten and Varopolous on the cogrowth of amenable groups of superpolynomial growth, is sufficient to prove our result.

The outline of this paper is as follows.  In \S\ref{sec:GP} we give a quick review of amenability and the Garrabant-Pak conjecture.  In \S\ref{sec:heart}, which is the heart of the paper, we use asymptotic analysis of $G$-functions near singularities to give a criterion for non-$D$-finiteness of a power series.  In \S\ref{sec:th}, we give the proof of Theorem \ref{thm:main1} and in \S\ref{sec:conc} we give some concluding remarks and pose some questions.

\section{Amenable groups and the Garrabant-Pak conjecture}
\label{sec:GP}
In this brief section we first recall what it means for a group to be amenable, and then
state the conjecture of Garrabant and Pak \cite{GP}. Let $G$ be a locally compact Hausdorff
group. Then $G$ is equipped with Haar measure and we let
$L^{\infty}(G)$ be the essentially bounded measurable functions on
$G$. Roughly speaking, amenable groups are groups for which the
Banach-Tarski paradox does not occur when equipped with Haar measure.
More formally, $G$ is \emph{amenable} if there is a linear functional
$\Lambda\in \operatorname{Hom}(L^{\infty}(G),\mathbb{R})$ of norm $1$
with the following properties. If $f\in L^{\infty}(g)$ is nonnegative
almost everywhere then $\Lambda(f)\ge 0$ and
$\Lambda(g\cdot f)=\Lambda(f)$ for $g\in L^{\infty}(G)$ and $g\in
G$. The action of $G$ on $L^{\infty}(G)$ is given by
$g\cdot f(x)=f(g^{-1}x)$.
 
%Kesten's Criterion
A remarkable result of Kesten~\cite{Kesten} states that for finitely generated groups, amenability can be characterized in terms of cogrowth (or, equivalently, in terms of probability of return): a finitely generated group $G$ with symmetric generating set $S$ is amenable if and only if ${\rm CL}(n;G,S)^{1/n} \to |S|$ as $n\to\infty$, or equivalently, if the probability of return, $p_{G,S}(n)$, satisfies $p_{G,S}(n)^{1/n}\to 1$.

Garrabant and Pak~\cite[Conjecture 13]{GP} make the following conjecture: \emph{``Let G be an amenable group of superpolynomial growth, and let $S$ be a symmetric generating set. Then the probability of return sequence  $p_{G,S}(n)$ is not $P$-recursive.''}  In addition, they prove their conjecture holds for the following classes of groups:
\begin{itemize}
\item virtually solvable groups of exponential growth with finite Pr\"ufer rank;
\item amenable linear groups of superpolynomial growth;
\item groups of weakly exponential growth
$A\exp(n^{\alpha}) <\gamma_{G,S}(n) < B\exp(n^{\beta})$
with $A, B > 0$, and $0 < \alpha,\beta < 1$, where $\gamma_{G,S}(n)$ is the number of distinct elements of $G$ that can be expressed as a product of $n$ elements of $S$; 
\item the Baumslag-Solitar groups $BS(k,1)$ with $k\ge 2$;
\item the lamplighter groups $L(d, H) = \mathbb{Z}^d \wr H$, where $H$ is a finite abelian group and $d\ge 1$.
\end{itemize}

To prove Theorem~\ref{thm:main1} we use basic results about singularities of $D$-finite power series.  Observe that a cogrowth generating function has nonnegative integer coefficients and it has positive, finite radius of convergence; in particular, it is a $G$-function.

\section{The asymptotic growth of $G$-functions}
\label{sec:heart}
In this section, we use the theory of the asymptotic behaviour of $G$-functions near a singularity to deduce necessary conditions for $D$-finiteness of a series.  We first recall some basic facts about $G$-functions.  If $F$ is a $G$-function then it is annihilated by a Fuchsian differential operator in the first Weyl algebra and consequently the singularities of $F$ are all regular.  Following work from Chudnovsky and Chudnovsky \cite{CC}, Katz \cite{Katz}, Andr\'e \cite{And1} (see Andr\'e \cite{And2} for a discussion) we have that if $\rho$ is a nonzero singularity of $F$ then if $L$ is a closed ray starting at $\rho$, then there is a simply connected open set $U$, containing $0$ and $\rho$, such that $F$ admits an analytic continuation on $U\setminus L$, and in some open neighbourhood of $\rho$ intersected with $U$ we have an expression
\begin{equation} \label{eq:F} 
F(z)=\sum_{i=1}^s \sum_{k=0}^{k_i} (\rho-z)^{\lambda_i}(\log(\rho-z))^k f_{i,k}(\rho-z),
\end{equation}
with $\lambda_1,\ldots ,\lambda_s$ rational numbers such that $\lambda_i-\lambda_j\not\in \mathbb{Z}$ for $i\neq j$ and such that $f_{i,k}(\rho-z)$ is analytic in a neighbourhood of $z=\rho$ and for each $i$ there is some $k$ such that $f_{i,k}(0)\neq 0$.  

%Given two functions having the form above for some $\rho$, we say that $F$ and $G$ are equivalent, if their difference, $F-G$, is analytic in a neighbourhood of $z=\rho$.  Notice that if some $\lambda_j$ is a nonnegative integer, then $G(z):=F(z)-(\rho-z)^{\lambda_j} f_{j,0}(z)$ is equivalent to $F(z)$ and has the additional property that
%$$G(z) =  \sum_{i=1}^s \sum_{k=0}^{k_i} (\rho-z)^{\lambda_i+m\delta_{i,j}}(\log(\rho-z))^k g_{i,k}(\rho-z),$$ where $g_{j,0}(z)\equiv 0$, $m$ is the minimum order of vanishing of $f_{j,1}(z),\ldots ,f_{j,k_j}(z)$ at $z=0$,
%$g_{i,k}(z)=f_{i,k}(z)$ for $i\neq j$ and for all applicable $k$, and $g_{j,k}(z) = f_{j,k}(z)/z^m$ for $k=1,\ldots , k_j$.  In this way we see that given 
%Given a decomposition of $F$ of the form given in Equation (\ref{eq:F}), we always have that $F$ is equivalent to something of the above form with $f_{j,0}(z)\equiv 0$ whenever $\lambda_j$ is a nonnegative integer.  We say that a function $F(z)$ of the form given in Equation (\ref{eq:F}) is \emph{reduced} if $f_{j,0}(z)\equiv 0$ whenever $\lambda_j$ is a nonnegative integer.  
Given $F(z)$ as in Equation (\ref{eq:F}), we define
\begin{equation} 
\Lambda(F;\rho) = \sum_{\{j\colon \lambda_j\not \in \mathbb{Z}\}} \lambda_j,
\end{equation}
where we take an empty sum to be $0$.

In order to better understand the situation, we first focus on the case where there is only a single $\lambda_i$ in the expression for $F$.  This is the content of the following lemma.
\begin{lemma}
Let $\lambda$ be a rational number that is not an integer and let $\rho$ be a nonzero complex number. Suppose that $F(z)$ is a nonzero function with the property that in some open set of the form $U\setminus L$ with $U$ and open neighbourhood of $z=\rho$ and $L$ a ray emanating from $z=\rho$, we have $$F(z) = \sum_{k=0}^{\ell} (\rho-z)^{\lambda} (\log(\rho-z))^k f_k(\rho-z),$$ where each $f_k$ is analytic at $z=\rho$ and $f_k(0)\neq 0$ for some $k$.
Then $\Lambda(G';\rho) =\Lambda(G;\rho)-1$. 
%There are two cases. Of $\lambda$ is a positive integer, then we can write \[G'= -(\rho-z)^{\lambda-1} f_1(\rho-z) + H(z),\] where 
%\[H(z) = \sum_{k=1}^{\ell} (\rho-z)^{\lambda-1} (\log(\rho-z))^k h_k(\rho-z).\] Here each $h_k(\rho-z)$ is analytic in a neighbourhood of $z=\rho$ and 
%$h_k(0)\neq 0$ for some $k\ge 1$ (i.e., $\Lambda(H)=\lambda-1$). 
%If however, $\lambda =0$ then we can write \[G'= -(\rho-z)^{-1} f_1(\rho-z) + H(z),\] where either $f_1(0)=0$ and $\Lambda(H)=-1$ or $f_1(0)\neq 0$.
\label{lem:X}
\end{lemma}
\begin{proof}
We have 
\begin{eqnarray*} F'(z) &=&  - \sum_{k=0}^{\ell} \lambda (\rho-z)^{\lambda-1} (\log(\rho-z))^k f_k(\rho-z)\\
&-&  \sum_{k=0}^{\ell}   k (\rho-z)^{\lambda-1} \log(\rho-z)^{k-1} f_k(\rho-z)\\
&-& \sum_{k=0}^{\ell} (\rho-z)^{\lambda} (\log(\rho-z))^k f_k'(\rho-z).
\end{eqnarray*}
From this, we immediately see that $\Lambda(F')\ge \lambda-1$.  By assumption, $f_k(0)\neq 0$ for some $k$, so pick the largest $k$ for which it is nonzero.  Then the coefficient of $(\rho-z)^{\lambda-1} (\log(\rho-z))^k$ in $F'(z)$ is $-\lambda f_k(\rho-z) - (k+1) f_{k+1}(\rho-z)-(\rho-z) f_k'(\rho-z)$, where we take $f_{\ell+1}(z)=0$.  Then we see that this coefficient is equal to $-\lambda f_k(0)\neq 0$ at $z=\rho$ and so we are done.
\end{proof}
\begin{proposition}
Let $\rho$ be a positive real number, let $s\ge 1$ and let $\lambda_1,\ldots ,\lambda_s$ be distinct rational numbers with $\lambda_i-\lambda_j\not\in\mathbb{Z}$ for $i\neq j$. Suppose that $$F(z)=\sum_{i=1}^s \sum_{k=0}^{k_j} (\rho-z)^{\lambda_i}(\log(\rho-z))^k f_{i,k}(\rho-z),$$ in an open set of the form $U\setminus L$, where $U$ is an open neighbourhood of $z=\rho$ and $L$ is a closed ray emanating from $z=\rho$ such that $0\not\in L$, where the $f_{i,k}$ are analytic in a neighbourhood of $z=\rho$. Then either $F(z)$ is analytic in a neighbourhood of $\rho$ or there is some $j\ge 0$ such that $\limsup_{t\to 1^{-}} |F^{(j)}(\rho\cdot t)|\to\infty$.
\label{lem:div}
\end{proposition}
\begin{proof} We may assume without loss of generality that 
$$\lambda_1<\lambda_2<\cdots <\lambda_s.$$
First suppose that $\lambda_1<0$ and that $\limsup_{t\to 1^{-}} |F(\rho t)|\not\to\infty$.
Then $F(t)(\rho(1-t))^{-\lambda_1}\to 0$ as $t\to 1^{-}$.  But
$$F(t)(\rho-t)^{-\lambda_1} = \sum_{k=0}^{k_1} (\log(\rho-t))^{k} f_{1,k}(\rho-t) + \sum_{i=2}^s \sum_{k=0}^{k_i} (\rho-t)^{\lambda_i-\lambda_1} \log(\rho-t)^{k} f_{i,k}(\rho-t).$$  Now $\lambda_i-\lambda_1>0$ for $i=2,\ldots ,s$ and each $f_{i,k}$ is analytic in a neighbourhood of $z=\rho$ and so 
$$(\rho(1-t))^{\lambda_i-\lambda_1} \log(\rho(1-t))^{k} f_{i,k}(\rho(1-t))\to 0$$ as $t\to 1^{-}$  In particular, since
$F(t)(\rho(1-t))^{-\lambda_1}\to 0$ as $t\to 1^{-}$, we see that 
$$ \sum_{k=0}^{k_1} (\log(\rho-t))^{k} f_{1,k}(\rho(1-t)) \to 0$$ as $t\to 1^{-}$.
  
Now let $c_k = f_{1,k}(0)$.  Then by assumption there is some $k$ such that $c_k\neq 0$.  Moreover, by L'H\^opital's rule we have $(\log(\rho(1-t)))^k (f_{1,k}(\rho(1-t))-c_k) \to 0$ as $t\to 1^{-}$ for all $k$.  Hence the fact that
$$ \sum_{k=0}^{k_1} (\log(\rho(1-t)))^{k} f_{1,k}(\rho(1-t)) \to 0$$ as $t\to 1^{-}$ gives that
$$ \sum_{k=0}^{k_1} c_k (\log(\rho(1-t)))^{k}  \to 0$$ as $t\to 1^{-}$.  Let $t=1-1/n$.  Then we have
$$\sum_{k=0}^{k_1} c_k(-1)^k (\log\, \rho+ \log\, n)^k \to 0$$ as $n\to \infty$, which is impossible since $\log n\to \infty$ as $n\to \infty$ and $c_k\neq 0$ for some $k$. 
 
So we may assume that $\lambda_1\ge 0$.  In particular, we may assume that $\Lambda(F)\ge 0$ is nonnegative and is zero if and only if $s=1$ and $\lambda_1$ is a nonnegative integer.  So suppose that the conclusion to the statement of the lemma does not hold and let $\alpha\ge 0$ denote the infimum of all $\Lambda(F)$ such that $F$ is of the form in the statement of the proposition and such that the conclusion to the statement of the proposition does not hold for $F$.  Then we may pick an $F$ for which the conclusion doesn't hold with $\Lambda(F) <\alpha +1/2$.  Then by Lemma \ref{lem:X} we have
$\Lambda(F') =\Lambda(F) -s+\epsilon$, where $\epsilon=1$ if some $\lambda_j$ is a nonnegative integer and is zero otherwise.
Notice that $\Lambda(F') < \alpha$ unless $s=1$ and $\epsilon=1$.  In particular if $(s,\epsilon)\neq (1,1)$, then $\Lambda(F')<\alpha$ and so by definition of $\alpha$, there is some $j$ such that 
$\limsup_{t\to 1^{-}} |(F')^{(j)}(\rho\cdot t)| = \infty$.  But we then have $\limsup_{t\to 1^{-}} |F^{(j+1)}(\rho \cdot t)| =\infty$, as desired.

Thus we may assume that $s=1$ and $\lambda_1$ is a nonnegative integer.  In particular, 
$$F(z) = \sum_{j=0}^m (\log(\rho-z))^j g_j(\rho-z),$$ for some functions $g_0,\ldots ,g_m$ that are analytic in a neighbourhood of $z=\rho$ with $g_m\not\equiv 0$.  We may also assume that $F(z)$ is not analytic in a neighbourhood of $z=\rho$ and thus $m\ge 1$.  We let $\kappa$ denote the minimum of the orders of vanishing of $g_1(\rho-z),\ldots  ,g_m(\rho-z)$ at $z=\rho$.  We prove the claim by induction on $\kappa$.  When $\kappa=0$, we have
$g_m(0)\neq 0$.  Then 
$$\lim_{n\to\infty} |F(\rho-1/n)|= \lim_{n\to \infty} \left|\sum_{j=0}^m \log(n)^j g_j(1/n)\right|= \lim_{n\to\infty} \left|\sum_{j=0}^m g_j(0) \log(n)^j \right|= \infty,$$ since $g_k(0)\neq 0$ for some $k\in \{1,\ldots ,m\}$.  Now suppose that the claim holds whenever $\kappa<p$ and that the minimum of the orders of vanishing of $g_1(\rho-z),\ldots ,g_m(\rho-z)$ at $z=\rho$ is $p$ with $p\ge 1$.  
Then
$$F'(z) = \sum_{j=0}^m (\log(\rho-z))^j \left( g_j'(\rho-z) - (j+1) g_{j+1}(\rho-z)(\rho-z)^{-1}\right),$$ where we take $g_{m+1}=0$.
Now by assumption there is some largest $k\le m$ with $k\ge 1$ for which $g_k(\rho-z)$ vanishes to order exactly $p$ at $z=\rho$.  
Then $g_k'(\rho-z) - (k+1)g_{k+1}(\rho-z)(\rho-z)^{-1}$ vanishes to order exactly $p-1$ since $g_{k+1}(\rho-z)$ vanishes to order at least $p+1$; moreover, the functions
 $g_j'(\rho-z) - (j+1) g_{j+1}(\rho-z)(\rho-z)^{-1}$ vanish to order at least $p-1$ for all $j$. It follows by the induction hypothesis that 
there is some $j$ such that 
$\limsup_{t\to 1^{-}} |(F')^{(j)}(t\rho )| = \infty$.  But we then have $\limsup_{t\to 1^{-}} |F^{(j+1)}(t\rho )| =\infty$, which completes the proof of the proposition.
\end{proof} 
As a corollary of Proposition \ref{lem:div}, we have a criterion for non-$D$-finiteness, which is applicable to many series coming from combinatorics.
\begin{corollary} Let $F(t)\in \mathbb{Z}[[t]]$ be a power series with nonnegative integer coefficients and suppose that $F(t)$ has radius of convergence $\rho>0$.  If $F(t)$ is $D$-finite then there is some $j\ge 0$ such that $\lim_{t\to \rho^{-}} |F^{(j)}(t)|\to\infty$. \label{cor}
\end{corollary}
\begin{proof}
Suppose that $F$ is $D$-finite. By Pringsheim's theorem $\rho$ is a singularity of $F(z)$.  Since $F$ has positive finite radius of convergence and has integer coefficients, we have $F$ is a $G$-function and as a result there is an open neighbourhood $U$ of $\rho$ and a closed ray $L$ emanating from $\rho$ such that on $U$ we have
$$F(z)=\sum_{i=1}^s \sum_{k=0}^{k_i} (\rho-z)^{\lambda_i}(\log(\rho-z))^k f_{i,k}(\rho-z),$$
where $s\ge 1$ and $\lambda_1,\ldots ,\lambda_s$ are distinct rational numbers with $\lambda_i-\lambda_j\not\in\mathbb{Z}$ for $i\neq j$, and the $f_{i,k}(\rho-z)$ are analytic in a neighbourhood of $z=\rho$.  Moreover, we may assume that for each $j$ there is some $k\in \{0,\ldots ,k_j\}$ such that $f_{j,k}(0)\neq 0$ and that $[0,\rho)\not\subseteq L$.  Since $F$ is not analytic at $z=\rho$, by Proposition~\ref{lem:div} we have that there is some $j\ge 0$ such that $\limsup_{t\to \rho^{-}} |F^{(j)}(t)|\to\infty$.  The result follows.
\end{proof}
We note that a $D$-finite power series has only finitely many singularities, therefore one could quickly deduce the conclusion of Corollary \ref{cor} if the following question had an affirmative answer.
\begin{question} Let $F(z)$ be a power series with nonnegative integer coefficients having positive finite radius of convergence $r$. If $F^{(j)}(z)$ converges absolutely on the circle of radius $r$ for every $j\ge 0$, does $F(z)$ admit the circle of radius $r$ as its natural boundary?
\end{question}
\section{Proof of Theorem \ref{thm:main1}}
\label{sec:th}
In this section, we give the proof of Theorem \ref{thm:main1}. To obtain the proof of Theorem \ref{thm:main1}, we need a remark.
\begin{remark}
\label{rem:S}
Let $G$ be a finitely generated group that is generated by a finite symmetric generating set $S$, and let $\kappa$ be a nonnegative real number.  If 
${\rm CL}(n;G,S)\ge |S|^n/n^{\kappa}$ for infinitely many $n$, then $G$ is virtually nilpotent.
\end{remark}
\begin{proof} Let $V(n)$ denote the number of distinct elements of $G$ that can be expressed as a product of at most $n$ elements of $S$. A theorem of Varopolous (see \cite{V+},~\cite[Theorem 2]{SC}) shows that if $V(n)\ge cn^d$ for all $n$ then we have that $p_{G,S}(n)\le C_1 n^{-d/2}$ for some positive constant $C_1$ and so ${\rm CL}(n;G,S) \le C_1 |S|^n/n^{d/2}$ for all $n$. Consequently if ${\rm CL}(n;G,S)\ge |S|^n/n^{\kappa}$ for infinitely many $n$, and we pick an integer $p$ such that $p>\kappa$ then for each positive constant $C_2$, we have $${\rm CL}(n;G,S) \ge C_2 |S|^n n^{p}$$ for infinitely many $n$.  In particular, if $c$ is a positive constant, we cannot have $V(n)\ge c n^{2p}$ for all $n$ by the above remarks.  It follows that $V(n) < c n^{2p}$ for infinitely many $n$ since if there were only finitely many such $n$, we could find a positive constant $c'>c$ such that 
$V(n) < c' n^{2p}$ for all $n$.  A strengthening of Gromov's theorem due to Wilkie and van den Dries~\cite{DW}, which shows that if there is a polynomial $P(x)$ and an infinite set of $n$ for which $V(n)< P(n)$ then $G$ must be virtually nilpotent, now gives the result.
\end{proof}
\begin{theorem} Let $G$ be a finitely generated amenable group that is generated by a finite symmetric subset $S$.  If $G$ is not virtually nilpotent then the cogrowth series of $G$ with respect to $S$ is not $D$-finite.
\end{theorem}
\begin{proof}
Let \[F(t)=\sum_{n\ge 0} {\rm CL}(n;G,S)t^n\] and suppose that $F(t)$ is $D$-finite and $G$ is amenable.  Then $F(t)$ has integer coefficients and by Kesten's 
criterion~\cite{Kesten} we have that $${\rm CL}(n,G,S)^{1/n}\to |S|.$$  In particular, the radius of convergence of 
$F$ is $\rho:=1/|S|$, $F$ is analytic inside the disc of radius $\rho$, and $F$ is a $G$-function.  By 
Pringsheim's theorem, $F$ has a singularity at $t=\rho$.  
Then by Proposition~\ref{lem:div} we see that there is some $j\ge 0$ such that $\lim_{t\to \rho^{-}} F^{(j)}(t)\to\infty$.  In particular, 
the sum $$F^{(j)}(t)=\sum_{n\ge 0} {\rm CL}(n;G,S)n(n-1)\cdots (n-j+1) t^{n-j}$$ diverges at $t=\rho=1/|S|$.  
This necessarily gives ${\rm CL}(n;G,S)n(n-1)\cdots (n-j+1)|S|^{-j+n} > |S|^{-j} /n^2$ for infinitely many $n$, since otherwise, the sum would converge at $t=\rho$.
Hence $${\rm CL}(n;G,S)>|S|^{n}/(n^2\cdot (n(n-1)\cdots (n-j+1))) \ge |S|^n/n^{j+2}$$ for infinitely many $n$.
But an application of Remark \ref{rem:S} with $\kappa=j+2$ gives that $G$ is virtually nilpotent.  The result follows. 
\end{proof}
\section{Concluding remarks}\label{sec:conc}
We showed that finitely generated amenable groups that are not virtually nilpotent have non-$P$-recursive associated cogrowth sequences.  It is natural to ask what happens in the virtually nilpotent case.  It is not difficult to show that virtually abelian groups have $P$-recursive cogrowth sequences.  The reason for this is that in the torsion free abelian case one can interpret the cogrowth generating function as a diagonal of a multivariate rational power series. Such a series is known to be $D$-finite.  Dealing with the virtually abelian case presents minor additional difficulties and can be dealt with by first fixing a free abelian subgroup $H$ of finite index and then, given a generating set $S$, determining the regular sublanguage of $S^*$ consisting of words in $S$ that are in $H$---this is relatively simple to compute.  Once one has this, it is not difficult to express the cogrowth generating function as a sum of diagonals.

In the non-virtually abelian case we do not know whether the cogrowth sequence for a virtually nilpotent group can be $P$-recursive.  In particular, the case of the Heisenberg group, of unipotent upper-triangular integer matrices is an interesting case to work out.  A question related to Stanley's conjecture~ \cite{Stan} concerning whether the generating function for ${2n\choose n}^d$ (as a function of $n$) is transcendental for $d\ge 2$ (which was solved by Flajolet \cite{Flaj} and Sharif and Woodcock \cite{SW}), is the question of whether a virtually nilpotent group that is not virtually cyclic must have transcendental cogrowth generating series.   The connection is the observation that ${2n\choose n}^d$ is precisely the cogrowth sequence for the group $\mathbb{Z}^d = \langle x_1,\ldots ,x_d ~
|~ x_ix_j=x_jx_i\rangle$ with $S=\{x_1^{\pm 1},\ldots ,x_d^{\pm 1}\}$.  In general, for a finitely generated virtually nilpotent group $G$, one can find a (finitely generated) normal nilpotent subgroup $N$ of finite index.  If we let $d_i$ denote the rank of the $i$-th group in the descending central series, then if
$${\rm GKdim}(N):=\sum_{i\ge 1} i d_i$$ is even, then work of Bass~\cite{Bass}, along with work of Varopoulous~\cite{V+}, and asymptotic results concerning coefficients of algebraic power series (see, for example, \cite{Jung}) show that the cogrowth generating function is not algebraic.  This applies, in particular to the Heisenberg group $H$, which has ${\rm GKdim}(H)=4$.   Kuksov~\cite{Kuksov2} showed that groups whose cogrowth generating function is rational are exactly the finite groups (although he uses a slightly different definition of cogrowth, the proof is easily modified), and these are precisely the groups where the language consisting of words on the generating set that are equal to the identity forms a regular language.  Thus, by analogy, one might guess that groups whose associated cogrowth series are all algebraic are necessarily virtually free, since these are the groups for which the word problem is context-free.  We thus pose the following question.
\begin{question} Let $G$ be a finitely generated group with finite symmetric generating set $S$.  If $\sum {\rm CL}(n;G,S)t^n$ is algebraic, is $G$ virtually free?
\end{question}

\section*{Acknowledgment}
We thank Igor Pak, Peter Paule, and the referee for many useful comments.

%Laurent Saloff-Coste.
%\newblock Probability on groups: random walks and invariant diffusions.
%\newblock {\em Notices Amer. Math. Soc.}, 48(9):968--977, 2001.

%\bibitem{Jung}
%R.~Jungen.
%\newblock Sur les s\'{e}ries de {T}aylor n'ayant que des singularit\'{e}s
 % alg\'{e}brico-logarithmiques sur leur cercle de convergence.
%\newblock {\em Comment. Math. Helv.}, 3(1):266--306, 1931.

%\end{thebibliography}

%\bibliography{cogrowth}
%\end{document}

\end{document}